 \documentclass[final,1p,11pt]{elsarticle}
 \pdfoutput=1
\makeatletter
\def\ps@pprintTitle{%
 \let\@oddhead\@empty
 \let\@evenhead\@empty
 \def\@oddfoot{\centerline{\thepage}}%
 \let\@evenfoot\@oddfoot}
\makeatother

\usepackage[T1]{fontenc}
\usepackage{charter}
\usepackage[expert]{mathdesign}
\usepackage{titlesec}
\titleformat{\section}[hang]{
    \usefont{T1}{bch}{b}{n}\selectfont} % "bch" - Bitstream Charter, "b" - bold
    {} % label
    {0em} % horizontal separation between label and title body
    {\hspace{-0.4pt}\Large \thesection\hspace{0.6em}} % code preceding the title
    [] % optional code following the title body

\usepackage{tocloft}

\usepackage[active]{srcltx}
\usepackage[all]{xy} \xyoption{poly}

\usepackage[T1]{fontenc}
\usepackage[latin9]{inputenc}
\usepackage{array}
\usepackage{longtable}
\usepackage{multirow}
\usepackage{amsmath,amsthm}
\usepackage{enumitem}
\usepackage{graphicx}
\usepackage{float}
\usepackage{xspace}
\usepackage{amssymb}
\usepackage[english]{babel}
\usepackage{ifthen}
\usepackage{amscd}
\usepackage{psfrag}
\usepackage[hyperindex,bookmarksnumbered, plainpages,backref]{hyperref}

\usepackage{natbib}

\makeatletter

\providecommand{\tabularnewline}{\\}

\theoremstyle{plain}
\newtheorem{thm}{\protect\theoremname}
  \theoremstyle{definition}
  \newtheorem{defn}[thm]{\protect\definitionname}
  \theoremstyle{plain}
  \newtheorem{lem}[thm]{\protect\lemmaname}
  \theoremstyle{remark}
  
  \theoremstyle{definition}
  \newtheorem{example}[thm]{\protect\examplename}
  \theoremstyle{plain}
  \newtheorem{cor}[thm]{\protect\corollaryname}

\AtBeginDocument{
  
}

\makeatother

\usepackage{babel}
  \providecommand{\corollaryname}{Corollary}
  \providecommand{\definitionname}{Definition}
  \providecommand{\examplename}{Example}
  \providecommand{\lemmaname}{Lemma}
  \providecommand{\remarkname}{Remark}
\providecommand{\theoremname}{Theorem}

\global\long\def\ka{\mathbf{k}}

\global\long\def\Aut{\operatorname*{Aut}}

\global\long\def\Ann{\operatorname*{Ann}}

\global\long\def\V{\operatorname*{V}}

\global\long\def\J{\mathcal{J}}

\global\long\def\T{\mathcal{T}}

\global\long\def\F{\mathcal{F}}

\global\long\def\G{\operatorname{G}}

\global\long\def\Z{\operatorname{Z}}

\global\long\def\GL{\operatorname{GL}}

\global\long\def\nil{\operatorname*{nilindex}}

\global\long\def\Hom{\operatorname*{Hom}}

\global\long\def\LieN{\operatorname*{LieN}}

\global\long\def\Assoc{\operatorname*{Assoc}}

\global\long\def\AssocN{\operatorname*{AssocN}}

\global\long\def\Jor{\operatorname*{\J\hspace{-0.063cm}or}}

\global\long\def\JorN{\operatorname*{\J\hspace{-0.063cm}orN}}

\begin{document}

\title{Geometric classification of nilpotent Jordan algebras of dimension
five}

\author[ime]{Iryna Kashuba\fnref{fn1}\corref{cor1}}
\ead{kashuba@ime.usp.br}

\author[ime]{Mar\'{\i}a Eugenia Martin\fnref{fn2}}
\ead{eugenia@ime.usp.br}

\cortext[cor1]{Corresponding author}
\fntext[fn1]{The author was supported by CNPq (309742/2013-7).}
\fntext[fn2]{The author was supported by the CAPES scholarship 
 for postdoc program of Mathematics, IME-USP.}

\address[ime]{Instituto de Matem\'atica e Estat\'\i stica, Universidade de S\~ao Paulo, R. do Mat\~ao 1010, 05508-090, S\~ao Paulo, Brazil.}

\begin{abstract} The variety $\JorN_{5}$ of five-dimensional nilpotent
Jordan algebras  structures over an algebraically closed field is investigated.
We show that $\JorN_{5}$ is the union of five irreducible components, four of them 
correspond to the Zariski closure of the $\GL_5$-orbits of four rigid algebras 
and the other one is the Zariski closure of an union of orbits of infinite family of 
algebras, none of them being rigid. 
\end{abstract}


\maketitle

\section{Introduction}

In this paper we study the geometry of nilpotent Jordan algebras over algebraically closed fields
$\ka$ of characteristic $\neq2,3$. Namely, let $V$ be an $n$-dimensional $\ka$-vector space, the bilinear 
maps $V\times V\to V$ form an $n^3$-dimensional algebraic variety $\Hom_{\ka}(V\otimes V,V)$. 
Moreover, Jordan algebra structures satisfy polynomial equations
(originating from both commutativity and Jordan identity) and therefore determine a Zariski-closed subset, 
which we will call the variety of Jordan algebras of dimension $n$. As any affine variety, $\Jor_{n}$ admits 
decomposition in a finite number of irreducible components. Furthermore, nilpotency conditions are also of 
polynomial type and thus the set of nilpotent Jordan algebras in dimension $n$, which will be denoted
by $\JorN_{n}$, possess a structure of algebraic variety embedded in $\Jor_{n}$. The general linear group
$\G=\GL(\V)$ acts on both $\Jor_{n}$ and $\JorN_{n}$ via ``change of basis'', decomposing them into $\G$-orbits which 
correspond to the classes of isomorphic Jordan algebras. We study the geometry of $\JorN(n)$: 
dimension, decomposition into irreducible components and their description. 
The crucial tool for it is the notion of deformation of algebras: 
we say that $\J_1$ is a deformation of $\J_2$
if the $\G$-orbit of $\J_2$ belongs to Zariski closure of the $\G$-orbit of 
$\J_1$. 
Algebra $\J$ whose orbit is an open subset in $\Jor_n$ is of particular 
interest, since it does not admit non-trivial deformations and its closure 
rises to irreducible component. We will call 
such algebras rigid.

Analogously to $\JorN_{n}$ one defines the varieties of nilpotent Lie and
associative algebras. We will denote them $\LieN_{n}$ and $\AssocN_{n}$,
respectively. The geometry of these two varieties is rather different.

In $1968$, F. Flanigan posed a question whether every irreducible component
is the closure of a $\G$-orbit of some rigid algebra $\J$.
In \cite{flanigan} he proved that every component of $\Assoc_{n}$ has to carry an open subset of
non-singular points which is either the orbit of a single rigid algebra
or an infinite union of orbits of algebras which differ only in the multiplication of the
radical. Flanigan has discovered that the second alternative does in fact
occur: there is a component of $\Assoc_{3}$ which consists
entirely of the orbits of an infinite family of three-dimensional
nilpotent algebras. 

In \cite{mazzolacocycles}, G.\,Mazzola has proved that 
the subvariety of commutative algebras inside $\AssocN_{n}$
is irreducible for $n\leq 6$ and it admits at least
two irreducible components in dimension $7$. For the non-commutative case the description is 
known only for $n\leq 5$. A.\,MakhLouf showed that $\AssocN_{2}$ is irreducible, while 
the number of irreducible components of $\AssocN_{n}$ when $n$ is equal to $3$, $4$ and $5$ is 
two, four and thirteen, respectively. Further, the author estimated a lower 
bound of 
the number of irreducible components in any dimension, see \cite{Makhlouf1993}. 

For the variety of nilpotent Lie algebras $\LieN_{n}$
the geometric classification is known for $n\leq 8$. In $1988$,
F. Grunewald and J. O'Halloran in \cite{Grunewald1988} established 
that $\LieN_{n}$ is irreducible when $n\leq 5$. Two years later,
S. Craig proved that this is the case for dimension $6$ as well, see 
\cite{Seeley6dim}.
The first reducible variety of nilpotent Lie algebras appears in dimension
$7$, see \cite{Goze1992}. It was shown by M. Goze and J. M. Ancochea-Bermudez. 
that $\LieN_{7}$ decomposes into two algebraic components each corresponding 
to infinite family of algebras. Finally,  in \cite{AncocheaBermudez199611} the
authors described all eight components of the variety of nilpotent
Lie algebras of dimension $8$.

As for the variety of Jordan algebras the references are rather recent. 
In \cite[2005]{irynashesta}, \cite[2014]{kashubamartin}
it was shown that in $\Jor_{n}$ for $n\leq4$ any irreducible component is determined by a rigid algebra.
In \cite[2011]{ancocheabermudes}, the authors establish that the
subvariety of nilpotent algebras $\JorN_{n}$ is irreducible
for $n\leq3$ and admits two irreducible components if $n=4$. 

In this paper, we will show that $\JorN_{5}$ is the union of five
irreducible components, there are four rigid algebras and the closure of the union of $\G$-orbits of a certain infinite family 
gives rise to the fifth component. 
 
The paper is organized as follows. Section \ref{sec:preliminaries} begins with some preliminaries on Jordan
algebras, further we establish the list of non-isomorphic nilpotent Jordan algebras of
dimension $5$. Among the non-isomorphic nilpotent Jordan algebras there are 
four 
infinite families depending on one or two parameters from $\ka$. In Section 
\ref{sec:conditions}, we define the variety of
Jordan algebras $\JorN_{n}$, determine useful invariants of deformation 
and show how to construct deformation between algebras or how to ensure
that such deformation does not exist.  In Section \ref{sec:examples}  we give 
several examples for the deformation both of and 
into the infinite families, while in Section \ref{sec:JorN5}, we prove the main 
result of this paper describing deformations
between the algebras in $\JorN_{5}$ and characterizing its the irreducible 
components.

\section{Jordan algebras: definition and basic concepts, list of nilpotent 
algebras up to dimension five\label{sec:preliminaries}}

In this section we present the basic concepts, notations for Jordan algebras and the
algebraic classification of nilpotent Jordan algebras of small dimensions. 
\begin{defn}
A \textbf{Jordan $\ka$-algebra} is a commutative $\ka$-algebra $\J$ satisfying
the Jordan identity, namely 
\begin{align}
((x\cdot x)\cdot y)\cdot x & =(x\cdot x)\cdot(y\cdot x), \qquad x,\,y\in\J.\label{eq:identidadejor}
\end{align}
\end{defn}
\noindent In a Jordan algebra $\J$ we define inductively a series of subsets
by setting 
\begin{gather*}
\J^{1}=\J\text{,}\\
\J^{n}=\J^{n-1}\cdot\J+\J^{n-2}\cdot\J^{2}+\cdots+\J\cdot\J^{n-1}.
\end{gather*}
The subset $\J^{n}$ is called the \textbf{$n$-th power of the algebra
$\J$}. 
\begin{defn}
A Jordan algebra $\J$ is said to be \textbf{nilpotent} if there exists
an integer $n\in\mathbb{N}$ such that $\J^{n}=0$. The least such
number is called the \textbf{index of nilpotency} of the algebra $\J$,
$\nil(\J)$.
\end{defn}

Further in this section we establish a list of all non-isomorphic nilpotent algebras of dimension five. 
For convenience we first describe all
indecomposable nilpotent algebras of dimension up to four, using the notations from \cite{martin}.
Also we drop $\cdot$ and denote a multiplication in $\J$ simply as $xy$. Products that we do not specify 
are understood to be $0$ or determined by commutativity.  

\begin{example}[\it Indecomposable nilpotent Jordan algebras of dimension one,  
and two three]
All nilpotent Jordan algebras of dimension less than four are associative. There is only one nilpotent 
indecomposable algebra in one and two: $\ka n$, with $n^{2}=0$  
and 
$\mathcal{B}_{3}$ generated by $n_{1}$, with $n_{1}^{3}=0$, respectively. There are two
indecomposable nilpotent algebras in dimension three, $\mathcal{T}_{3}$ generated by $n_{1}$,
such that $n_{1}^{4}=0$ and $\mathcal{T}_{4}=\langle n_1,n_2\,|\, n_1^2=n_2^2\rangle$.
\end{example}

In \cite{martin} all Jordan algebras of dimension four over an algebraically 
closed field of characteristic not $2$ were described. In the Table \ref{tab:4dnilpo} we list eight 
non-isomorphic indecomposable nilpotent Jordan algebras. 

\begin{longtable}{|>{\centering}m{1cm}|>{\centering}m{9cm}|}
\caption{\label{tab:4dnilpo}Indecomposable four-dimensional nilpotent Jordan
algebras.}
\tabularnewline
\hline 
$\F$ & Multiplication Table \tabularnewline
\hline 
\endhead
$\mathcal{\F}_{61}$ & $n_{1}^{2}=n_{2}\quad n_{2}^{2}=n_{4}\quad 
n_{1}n_{2}=n_{3}\quad n_{1}n_{3}=n_{4}$ \tabularnewline
\hline 
$\mathcal{\F}_{62}$ & $n_{1}^{2}=n_{2}\quad 
n_{4}^{2}=n_{2}\quad n_{1}n_{2}=n_{3}$  \tabularnewline
\hline 
$\mathcal{\F}_{63}$ & $n_{1}^{2}=n_{2}\quad 
n_{4}^{2}=-n_{2}-n_{3}\quad n_{1}n_{2}=n_{3}\quad n_{2}n_{4}=n_{3}$ 
\tabularnewline
\hline 
$\mathcal{\F}_{64}$ & $n_{1}^{2}=n_{2}\quad 
n_{4}^{2}=-n_{2}\quad n_{1}n_{2}=n_{3}\quad n_{2}n_{4}=n_{3}$ \tabularnewline
\hline 
$\mathcal{\F}_{65}$ & $n_{1}^{2}=n_{2}\quad 
n_{1}n_{2}=n_{3}\quad n_{2}n_{4}=n_{3}$ \tabularnewline
\hline 
$\mathcal{\F}_{66}$ & $n_{1}^{2}=n_{2}\quad 
n_{4}^{2}=n_{3}\quad n_{1}n_{2}=n_{3}$ \tabularnewline
\hline 
$\mathcal{\F}_{69}$ & $n_{1}^{2}=n_{2}\quad n_{1}n_{3}=n_{4}$ \tabularnewline
\hline 
$\mathcal{\F}_{70}$ & $n_{1}^{2}=n_{2}\quad n_{3}n_{4}=n_{2}$ \tabularnewline
\hline 
\end{longtable}

\subsection{Nilpotent Jordan algebras of dimension five\label{sub:Nilpotent5}}

In \cite{mazzolacocycles}, G. Mazzola has classified nilpotent commutative associative (and thus Jordan)
five-dimensional algebras. Keeping the notations from the paper we list the 25 
non-isomorphic such algebras in Table \ref{tab:5dnilpoassoc}. For each algebra 
we calculate the dimension of its automorphism
group $\Aut(\J)$, of its annihilator $\Ann(\J)=\left\{ a\in\J\mid a\J=0\right\} $,
of its second power $\J^{2}$ and the index of nilpotency $\nil(\J)$.

\begin{longtable}{|>{\centering}m{0.8cm}|>{\centering}m{5.8cm}|>{\centering}m{0.8cm}|>{\centering}m{0.8cm}|>{\centering}m{0.8cm}|>{\centering}m{1cm}|}
\caption{Five-dimensional nilpotent associative Jordan
algebras.}\label{tab:5dnilpoassoc}
\tabularnewline
\hline 
$\epsilon$ & Multiplication Table & $\dim$

$\Aut$ & $\dim$

$\Ann$ & $\dim$

$\epsilon^{2}$ & $\operatorname{nil-}$

$\operatorname{index}$\tabularnewline
\endhead
\hline 
$\epsilon_{1}$ & $n_{1}^{2}=n_{2}\quad n_{2}^{2}=n_{1}n_{3}=n_{4}$

$n_{1}n_{2}=n_{3}\quad n_{1}n_{4}=n_{2}n_{3}=n_{5}$ & $5$ & $1$ & $4$ & $6$\tabularnewline
\hline 
$\epsilon_{2}$ & $n_{1}^{2}=n_{2}\quad n_{2}^{2}=n_{1}n_{3}=n_{4}$

$n_{1}n_{2}=n_{3}$ & $6$ & $1$ & $3$ & $5$\tabularnewline
\hline 
$\epsilon_{3}$ & $n_{1}^{2}=n_{2}\quad n_{4}^{2}=n_{5}$

$n_{1}n_{2}=n_{4}n_{5}=n_{3}$ & $6$ & $1$ & $3$ & $4$\tabularnewline
\hline 
$\epsilon_{4}$ & $n_{1}^{2}=n_{3}\quad n_{1}n_{2}=n_{4}$

$n_{1}n_{4}=n_{2}n_{3}=n_{5}$ & $7$ & $1$ & $3$ & $4$\tabularnewline
\hline 
$\epsilon_{5}$ & $\mathcal{\F}_{61}\oplus\ka n_{5}$ & $7$ & $2$ & $3$ & $5$\tabularnewline
\hline 
$\epsilon_{6}$ & $\mathcal{T}_{3}\oplus\mathcal{B}_{3}$ & $7$ & $2$ & $3$ & $4$\tabularnewline
\hline 
$\epsilon_{7}$ & $n_{1}n_{3}=n_{4}\quad n_{2}n_{3}=n_{5}$

$n_{1}n_{2}=n_{4}-n_{5}$ & $7$ & $2$ & $2$ & $3$\tabularnewline
\hline 
$\epsilon_{8}$ & $n_{1}^{2}=n_{3}\quad n_{2}^{2}=n_{1}n_{3}=n_{4}$

$n_{1}n_{2}=n_{5}$ & $8$ & $2$ & $3$ & $4$\tabularnewline
\hline 
$\epsilon_{9}$ & $n_{1}n_{3}=n_{5}$

$n_{1}n_{2}=n_{4}\quad n_{2}n_{3}=-n_{5}$ & $8$ & $2$ & $2$ & $3$\tabularnewline
\hline 
$\epsilon_{10}$ & $n_{1}^{2}=n_{3}\quad n_{1}n_{3}=n_{4}\quad n_{1}n_{2}=n_{5}$ & $9$ & $2$ & $3$ & $4$\tabularnewline
\hline 
$\epsilon_{11}$ & $n_{1}^{2}=n_{2}n_{3}=n_{4}\quad n_{1}n_{3}=n_{5}$ & $9$ & $2$ & $2$ & $3$\tabularnewline
\hline 
$\epsilon_{12}$ & $n_{1}n_{2}=n_{4}\quad n_{1}n_{3}=n_{5}$ & $11$ & $2$ & $2$ & $3$\tabularnewline
\hline 
$\epsilon_{13}$ & $n_{3}^{2}=n_{4}\quad n_{1}n_{2}=n_{3}n_{4}=n_{5}$ & $8$ & $1$ & $2$ & $4$\tabularnewline
\hline 
$\epsilon_{14}$ & $\mathcal{\F}_{66}\oplus\ka n_{5}$ & $9$ & $2$ & $2$ & $4$\tabularnewline
\hline 
$\epsilon_{15}$ & $\mathcal{T}_{4}\oplus\mathcal{B}_{3}$ & $9$ & $2$ & $2$ & $3$\tabularnewline
\hline 
$\epsilon_{16}$ & $n_{1}^{2}=n_{3}\quad n_{2}^{2}=n_{5}\quad n_{1}n_{2}=n_{4}$ & $10$ & $3$ & $3$ & $3$\tabularnewline
\hline 
$\epsilon_{17}$ & $n_{1}^{2}=n_{4}\quad n_{3}^{2}=n_{1}n_{2}=n_{5}$ & $10$ & $2$ & $2$ & $3$\tabularnewline
\hline 
$\epsilon_{18}$ & $\mathcal{T}_{3}\oplus\ka n_{4}\oplus\ka n_{5}$ & $11$ & $3$ & $2$ & $4$\tabularnewline
\hline 
$\epsilon_{19}$ & $\mathcal{B}{}_{3}\oplus\mathcal{B}{}_{3}\oplus\ka n_{5}$ & $11$ & $3$ & $2$ & $3$\tabularnewline
\hline 
$\epsilon_{20}$ & $\mathcal{\F}_{69}\oplus\ka n_{5}$ & $12$ & $3$ & $2$ & $3$\tabularnewline
\hline 
$\epsilon_{21}$ & $n_{2}n_{3}=n_{1}n_{4}=n_{5}$ & $11$ & $1$ & $1$ & $3$\tabularnewline
\hline 
$\epsilon_{22}$ & $\mathcal{\F}_{70}\oplus\ka n_{5}$ & $12$ & $2$ & $1$ & $3$\tabularnewline
\hline 
$\epsilon_{23}$ & $\mathcal{T}_{4}\oplus\ka n_{4}\oplus\ka n_{5}$ & $14$ & $3$ & $1$ & $3$\tabularnewline
\hline 
$\epsilon_{24}$ & $\mathcal{B}{}_{3}\oplus\ka n_{3}\oplus\ka n_{4}\oplus\ka n_{5}$ & $17$ & $4$ & $1$ & $3$\tabularnewline
\hline 
$\epsilon_{25}$ & $\ka n_{1}\oplus\ka n_{2}\oplus\ka n_{3}\oplus\ka n_{4}\oplus\ka n_{5}$ & $25$ & $5$ & $0$ & $2$\tabularnewline
\hline 
\end{longtable}

Recently, in \cite{nilpodim5Arxiv} a method to construct nilpotent Jordan  algebras using central extension of lower dimension algebras was presented. The authors showed that any nilpotent algebra can be obtained via this methods and  
provided a list of $35$ single non-associative Jordan algebras together with $6$ families of algebras 
depending on parameters in $\ka$. 
The list of these algebras is given in Table \ref{tab:5dnilponaoassoc} (using 
the notations \cite{nilpodim5Arxiv})
together with the dimensions of $\Aut(\J)$, $\Ann(\J)$, $\J^{2}$ and of the 
associative center  $\Z(\J)=\{a\in\J\mid(a,\J,\J)=(\J,a,\J)=(\J,\J,a)=0\}$.

\begin{longtable}{|>{\centering}m{1cm}|>{\centering}m{6cm}|>{\centering}m{0.7cm}|>{\centering}m{0.7cm}|>{\centering}m{0.6cm}|>{\centering}m{0.6cm}|}
\caption{Five-dimensional nilpotent non-associative
Jordan algebras.}\label{tab:5dnilponaoassoc}
\tabularnewline
\hline 
$\J$ & Multiplication Table & $\dim$

$\Aut$ & $\dim$

$\Ann$ & $\dim$

$\J^{2}$ & $\dim$

$\Z$\tabularnewline
\endhead
\hline 
$\J_{1}$ & $\F_{65}\oplus\ka n_{5}$ & $8$ & $2$ & $2$ & $3$\tabularnewline
\hline 
$\J_{2}$ & $\F_{64}\oplus\ka n_{5}$ & $9$ & $2$ & $2$ & $3$\tabularnewline
\hline 
$\J_{3}$ & $\F_{63}\oplus\ka n_{5}$ & $8$ & $2$ & $2$ & $3$\tabularnewline
\hline 
$\J_{4}$ & $\F_{62}\oplus\ka n_{5}$  & $7$ & $2$ & $2$ & $3$\tabularnewline
\hline 
$\J_{5}$ & $n_{1}^{2}=n_{2}\quad n_{4}^{2}=n_{2}n_{3}=n_{5}$ & $6$ & $1$ & $2$ & $3$\tabularnewline
\hline 
$\J_{6}$ & $n_{1}^{2}=n_{2}\quad n_{1}n_{4}=n_{2}n_{3}=n_{5}$ & $7$ & $1$ & $2$ & $3$\tabularnewline
\hline 
$\J_{7}$ & $n_{1}n_{2}=n_{3}\quad n_{3}n_{4}=n_{5}$ & $6$ & $1$ & $2$ & $2$\tabularnewline
\hline 
$\J_{8}$ & $n_{1}n_{2}=n_{3}\quad n_{3}n_{4}=n_{1}^{2}=n_{5}$ & $5$ & $1$ & $2$ & $2$\tabularnewline
\hline 
$\J_{9}$ & $n_{1}n_{2}=n_{3}$

$n_{3}n_{4}=n_{1}^{2}=n_{2}^{2}=n_{5}$ & $4$ & $1$ & $2$ & $2$\tabularnewline
\hline 
$\J_{10}$ & $n_{1}n_{2}=n_{3}$

$n_{1}n_{3}=n_{4}^{2}=n_{2}n_{3}=n_{5}$ & $5$ & $1$ & $2$ & $3$\tabularnewline
\hline 
$\J_{11}$ & $n_{1}n_{2}=n_{3}\quad n_{1}n_{3}=n_{4}^{2}=n_{5}$ & $7$ & $1$ & $2$ & $3$\tabularnewline
\hline 
$\J_{12}$ & $n_{1}n_{2}=n_{3}$

$n_{1}n_{3}=n_{4}^{2}=n_{2}^{2}=n_{5}$ & $6$ & $1$ & $2$ & $3$\tabularnewline
\hline 
$\J_{13}$ & $n_{1}n_{2}=n_{3}$

$n_{1}n_{4}=n_{1}n_{3}=n_{2}n_{3}=n_{5}$ & $6$ & $1$ & $2$ & $3$\tabularnewline
\hline 
$\J_{14}$ & $n_{1}n_{2}=n_{3}\quad n_{1}n_{3}=n_{2}n_{4}=n_{5}$ & $7$ & $1$ & 
$2$ & $3$\tabularnewline
\hline 
$\J_{15}^{\alpha}$ & $n_{1}^{2}=n_{2}\quad n_{1}n_{2}=n_{3}\quad n_{4}^{2}=\alpha n_{5}$

$n_{2}^{2}=n_{1}n_{3}=n_{2}n_{4}=n_{5},\,{\scriptstyle \alpha\in\ka}$ & $6$ & $1$ & $3$ & $3$\tabularnewline
\hline 
$\J_{16}$ & $n_{1}n_{2}=n_{3}^{2}=n_{4}\quad n_{1}n_{4}=n_{5}$ & $6$ & $1$ & $2$ & $2$\tabularnewline
\hline 
$\J_{17}$ & $n_{1}n_{2}=n_{3}^{2}=n_{4}$

$n_{1}n_{4}=n_{2}n_{3}=n_{5}$ & $5$ & $1$ & $2$ & $2$\tabularnewline
\hline 
$\J_{18}$ & $n_{1}n_{2}=n_{3}^{2}=n_{4}$

$n_{1}n_{4}=n_{2}^{2}=n_{5}$ & $4$ & $1$ & $2$ & $2$\tabularnewline
\hline 
$\J_{19}$ & $n_{1}n_{2}=n_{3}^{2}=n_{4}\quad n_{3}n_{4}=n_{5}$ & $5$ & $1$ & $2$ & $2$\tabularnewline
\hline 
$\J_{20}$ & $n_{1}n_{2}=n_{3}^{2}=n_{4}\quad n_{3}n_{4}=n_{1}^{2}=n_{5}$ & $4$ & $1$ & $2$ & $2$\tabularnewline
\hline 
$\J_{21}$ & $n_{1}n_{2}=n_{3}^{2}=n_{4}$

$n_{3}n_{4}=n_{1}^{2}=n_{2}^{2}=n_{5}$ & $3$ & $1$ & $2$ & $2$\tabularnewline
\hline 
$\J_{22}$ & $n_{1}^{2}=n_{2}\quad n_{1}n_{2}=n_{3}^{2}=n_{4}$

$n_{2}^{2}=n_{1}n_{4}=n_{5}$ & $4$ & $1$ & $3$ & $3$\tabularnewline
\hline 
$\J_{23}^{\beta}$ & $n_{1}^{2}=n_{3}\quad n_{1}n_{3}=n_{2}n_{4}=n_{5}$

$n_{1}n_{2}=n_{4}\quad n_{2}n_{3}=\beta n_{5},\:{\scriptstyle \beta\in\ka}$ & $6$ & $1$ & $3$ & $3$\tabularnewline
\hline 
$\J_{24}^{\gamma}$ & $n_{1}^{2}=n_{3}\quad n_{1}n_{2}=n_{4}\quad n_{2}n_{3}=n_{5}$

$n_{1}n_{4}=\gamma n_{5},\:{\scriptstyle \gamma\in\ka^{*}-\{1,-\frac{1}{2}\}}$  
& $7$ & $1$ & $3$ & $3$\tabularnewline
\hline 
$\J_{24}^{0}$ & $n_{1}^{2}=n_{3}\quad n_{1}n_{2}=n_{4}\quad 
n_{2}n_{3}=n_{5}$  & $7$ & $2$ & $3$ & $3$\tabularnewline
\hline 
$\J_{24}^{-\frac{1}{2}}$ & $n_{1}^{2}=n_{3}\quad n_{1}n_{2}=n_{4}\quad n_{2}n_{3}=n_{5}$

$n_{1}n_{4}=-\frac{1}{2}n_{5}$  & $8$ & $1$ & $3$ & $3$\tabularnewline

\hline 
$\J_{25}$ & $n_{1}^{2}=n_{3}\quad n_{1}n_{2}=n_{4}$

$n_{1}n_{3}=n_{1}n_{4}=n_{5}\quad n_{2}n_{3}=-2n_{5}$ & $7$ & $1$ & $3$ & $3$\tabularnewline
\hline 
$\J_{26}^{\delta}$ & $n_{1}^{2}=n_{3}\quad n_{2}^{2}=n_{4}\quad n_{2}n_{4}=\delta n_{5}$

$n_{1}n_{3}=n_{1}n_{4}=n_{5},\:{\scriptstyle \delta\in\ka^{*}}$ & $6$ & $1$ & $3$ & $3$\tabularnewline
\hline 
$\J_{26}^{0}$ & $n_{1}^{2}=n_{3}\quad n_{2}^{2}=n_{4}$

$n_{1}n_{3}=n_{1}n_{4}=n_{5}$ & $6$ & $2$ & $3$ & $3$\tabularnewline
\hline 
$\J_{27}^{\varepsilon,\phi}$ & $n_{1}^{2}=n_{3}\quad n_{2}n_{3}=n_{1}n_{4}=n_{5}$

$n_{2}^{2}=n_{4}\quad n_{2}n_{4}=\varepsilon n_{5}$

$n_{1}n_{3}=\phi n_{5}\:{\scriptstyle \varepsilon,\phi\in\ka,\:\varepsilon\phi\neq1}$ & $6$ & $1$ & $3$ & $3$\tabularnewline
\hline 
$\J_{27}^{\varepsilon,\varepsilon^{-1}}$ & $n_{1}^{2}=n_{3}\quad n_{2}n_{3}=n_{1}n_{4}=n_{5}$

$n_{2}^{2}=n_{4}\quad n_{2}n_{4}=\varepsilon n_{5}$

$n_{1}n_{3}=\varepsilon^{-1}n_{5}\quad{\scriptstyle 
\varepsilon\in\ka^{*}-\{-1\}}$ & $6$ & $2$ & $3$ & 
$3$\tabularnewline
\hline 
$\J_{27}^{-1,-1}$ & $n_{1}^{2}=n_{3}\quad 
n_{2}n_{3}=n_{1}n_{4}=n_{5}$

$n_{2}^{2}=n_{4}\quad n_{2}n_{4}=n_{1}n_{3}=-n_{5}$ & 
$6$ & $2$ & $3$ & $3$\tabularnewline
\hline 
$\J_{28}$ & $n_{1}^{2}=n_{2}\quad n_{2}n_{3}=n_{4}\quad n_{3}^{2}=n_{5}$ & $7$ & $2$ & $3$ & $3$\tabularnewline
\hline 
$\J_{29}$ & $n_{1}^{2}=n_{2}\quad n_{2}n_{3}=n_{4}\quad n_{1}n_{3}=n_{5}$ & $7$ 
& $2$ & $3$ & $3$\tabularnewline
\hline 
$\J_{30}$ & $n_{1}^{2}=n_{2}\quad n_{2}n_{3}=n_{4}$

$n_{1}n_{3}=n_{3}^{2}=n_{5}$ & $6$ & $2$ & $3$ & $3$\tabularnewline
\hline 
$\J_{31}$ & $n_{1}^{2}=n_{2}\quad n_{2}n_{3}=n_{4}\quad n_{1}n_{2}=n_{5}$ & $7$ & $2$ & $3$ & $3$\tabularnewline
\hline 
$\J_{32}$ & $n_{1}^{2}=n_{2}\quad n_{2}n_{3}=n_{4}$

$n_{1}n_{2}=n_{1}n_{3}=n_{5}$ & $6$ & $2$ & $3$ & $3$\tabularnewline
\hline 
$\J_{33}$ & $n_{1}^{2}=n_{2}\quad n_{2}n_{3}=n_{4}$

$n_{1}n_{2}=n_{3}^{2}=n_{5}$ & $5$ & $2$ & $3$ & $3$\tabularnewline
\hline 
$\J_{34}$ & $n_{1}n_{2}=n_{3}\quad n_{1}n_{3}=n_{4}\quad n_{1}^{2}=n_{5}$ & $8$ & $2$ & $3$ & $3$\tabularnewline
\hline 
$\J_{35}$ & $n_{1}n_{2}=n_{3}\quad n_{1}n_{3}=n_{4}\quad n_{2}^{2}=n_{5}$ & $7$ & $2$ & $3$ & $3$\tabularnewline
\hline 
$\J_{36}$ & $n_{1}n_{2}=n_{3}\quad n_{1}n_{3}=n_{4}$

$n_{2}n_{3}=n_{5}$ & $6$ & $2$ & $3$ & $3$\tabularnewline
\hline 
$\J_{37}$ & $n_{1}^{2}=n_{2}^{2}=n_{5}\quad n_{1}n_{2}=n_{3}$ \quad 
$ n_{1}n_{3}=n_{4}$ & $6$ & $2$ & $3$ & 
$3$\tabularnewline
\hline 
$\J_{38}$ & $n_{1}n_{2}=n_{3}\quad n_{1}n_{3}=n_{4}$

$n_{1}^{2}=n_{2}n_{3}=n_{5}$ & $5$ & $2$ & $3$ & $3$\tabularnewline
\hline 
$\J_{39}$ & $n_{1}n_{2}=n_{3}\quad n_{1}^{2}=n_{5}$

$n_{2}^{2}=n_{1}n_{3}=n_{4}$ & $7$ & $2$ & $3$ & $3$\tabularnewline
\hline 
$\J_{40}$ & $n_{1}n_{2}=n_{3}\quad n_{2}^{2}=n_{1}n_{3}=n_{4}$

$n_{1}^{2}=n_{2}n_{3}=n_{5}$ & $4$ & $2$ & $3$ & $3$\tabularnewline
\hline 
$\J_{41}^{\lambda}$ & $n_{1}^{2}=n_{5}\quad n_{1}n_{2}=n_{3}\quad 
n_{2}^{2}=\lambda n_{5}$

$n_{1}n_{3}=n_{2}n_{3}=n_{4},\:{\scriptstyle \lambda\in\ka}$ & $6$ & $2$ & $3$ & 
$3$\tabularnewline
\hline 
\end{longtable}

\begin{thm}\label{thm:classification}
In Tables \ref{tab:5dnilpoassoc} and \ref{tab:5dnilponaoassoc}  all algebras 
are one-to-one non-isomorphic except for
\begin{enumerate}
\item\label{family-one} $\J_{15}^{\alpha_{1}}\simeq\J_{15}^{\alpha_{2}}$, 
$\J_{24}^{\alpha_{1}}\simeq\J_{24}^{\alpha_{2}}$ if and only if 
$\alpha_1=\alpha_2$;
\item\label{family-two} $\J_{23}^{\beta_{1}}\simeq\J_{23}^{\beta_{2}}$, 
$\J_{26}^{\beta_{1}}\simeq\J_{26}^{\beta_{2}}$ if and only if 
$\beta_1=\pm\beta_2$;
\item\label{family-three} 
$\J_{27}^{\alpha_{1},\beta_{1}}\simeq\J_{27}^{\alpha_{2},\beta_{2}}$ if and 
only if $(\alpha_{1},\beta_{1})=(\alpha_{2},\beta_{2})$
or $(\alpha_{1},\beta_{1})=(\beta_{2},\alpha_{2})$;
\item\label{family-four} $\J_{41}^{\lambda_{1}}\simeq\J_{41}^{\lambda_{2}}$ if 
and only if 
$\lambda_1=\lambda_2$ or $\lambda_1=\lambda_{2}^{-1}$;
\item\label{family-five} $\J^{1}_{24}$ is in fact associative and corresponds 
to  $\epsilon_4$ from Table\,\ref{tab:5dnilpoassoc};
\item\label{family-six} $\J_{29}\simeq \J^0_{24}$; 
\item\label{family-seven} $\J_{37}\simeq \J_{27}^{-1,-1}$.
\item\label{family-eight} The family $\J_{41}^{\lambda}$ is isomorphic to 
$\J_{27}^{\varepsilon,\varepsilon^{-1}}$ for
$\lambda\neq 0,1$ and 
$\varepsilon=-\frac{(-1+\sqrt{\lambda})^{2}}{(1+\sqrt{\lambda})^{2}}$, while 
$\J_{41}^{1}$ is isomorphic to $\J_{26}^{0}$. Remaining only $\J_{41}^{0}$, 
which will be denoted simply as $\J_{41}$.
\end{enumerate}
\end{thm}

\begin{proof}
Firstly, in \cite{nilpodim5Arxiv} the authors showed the isomorphism from items 
\ref{family-one}-\ref{family-four}.
The isomorphism \ref{family-five} follows immediately from 
the multiplication tables of the corresponding algebras. To see 
\ref{family-six}-\ref{family-eight} consider the change of basis 
\begin{center}
\begin{tabular}{l|l|l}
$\J_{29}\simeq \J_{24}^{0}$ & $\J_{37}\simeq \J_{27}^{-1,-1}$ & 
$\J_{41}^{1}\simeq \J_{26}^{0}$ \tabularnewline
$e_{1}=n_{1}$ & $e_{1}=n_{1}-n_{2}$ & $e_{1}=n_{1}+n_{2}$\tabularnewline
$e_{2}=n_{3}$ & $e_{2}=-n_{1}-n_{2}$ & $e_{2}=in_{1}-in_{2}$\tabularnewline
$e_{3}=n_{2}$ & $e_{3}=-2(n_{3}-n_{5})$ & $e_{3}=2(n_{3}+n_{5})$\tabularnewline
$e_{4}=n_{5}$ & $e_{4}=2(n_{3}+n_{5})$ & $e_{4}=2(n_{3}-n_{5})$\tabularnewline
$e_{5}=n_{4}$ & $e_{5}=2n_{4}$ & $e_{5}=4n_{4}$\tabularnewline
\end{tabular}
\end{center}

and for any $\lambda\neq 0,1$ consider 
$\varepsilon=-\frac{(-1+\sqrt{\lambda})^{2}}{(1+\sqrt{\lambda})^{2}}$ then 
the transformation 
\begin{align*}
e_{1} & =\sqrt{\lambda}n_{1}+n_{2}\\
e_{2} & 
=\left(\frac{\sqrt{\lambda}-\lambda}{1+\sqrt{\lambda}}\right)n_{1}+\left(\frac{
-1+\sqrt{\lambda}}{1+\sqrt{\lambda}}\right)n_{2}\\
e_{3} & =2\sqrt{\lambda}n_{3}+2\lambda n_{5}\\
e_{4} & 
=-\frac{2(-1+\sqrt{\lambda})^{2}\sqrt{\lambda}}{(1+\sqrt{\lambda})^{2}}n_{3}
+\frac{2(-1+\sqrt{\lambda})^{2}\lambda}{(1+\sqrt{\lambda})^{2}}n_{5}\\
e_{5} & =-\frac{2(-1+\sqrt{\lambda})^{2}\sqrt{\lambda}}{1+\sqrt{\lambda}}n_{4}
\end{align*}
gives the isomorphism 
$\J_{41}^{\lambda}\simeq\J_{27}^{\varepsilon,\varepsilon^{-1}}$. 

Further the dimension of the automorphism group, the annihilator, any power of 
algebra, the associative center, index of nilpotency  together with associative 
and non-associative conditions, are invariants of class of isomorphic 
algebras. Therefore any two algebras whose last four columns in 
Table\,\ref{tab:5dnilpoassoc} or  Table\,\ref{tab:5dnilponaoassoc}  are not 
equal are non-isomorphic, and one has to verify
that the algebras are non-isomorphic in the following cases:
\begin{enumerate}
\item $\J_{9}-\J_{18}-\J_{20}$: the four-dimensional non-associative nilpotent
Jordan algebra $\F_{64}$ is a subalgebra of $\J_{20}$ but it is not
a subalgebra nor of $\J_{9}$ neither of $\J_{18}$. Analogously,
$\F_{70}$ is a subalgebra of $\J_{9}$ and is not of $\J_{18}$.
\item $\J_{8}-\J_{17}-\J_{19}$: the dimension of the second cohomology
group of $\J_{19}$ with coefficients in $\J_{19}$, $H^{2}(\J_{19},\J_{19})$,
is $6$ while $\dim H^{2}(\J_{8},\J_{8})=\dim H^{2}(\J_{17},\J_{17})=5$.
For the definition of the cohomology group for Jordan algebras we
refer to \cite{kashubamartin}. On the other hand,
the four-dimensional algebra $\F_{63}$
is a subalgebra of $\J_{8}$ but is not of $\J_{17}$.
\item $\J_{33}\not\simeq\J_{38}$: the three-dimensional nilpotent Jordan algebra
$\T_{4}$ is a subalgebra of $\J_{38}$ but is not of $\J_{33}$.
\item $\J_{5}-\J_{12}-\J_{13}$: $\dim H^{2}(\J_{12},\J_{12})=\dim H^{2}(\J_{13},\J_{13})=13$
 while $\dim H^{2}(\J_{5},\J_{5})=12$
But the dimension of Jacobi space of $\J_{12}$, i.e. 
$$\operatorname{Jac}(\J)=\left\{ a\in\J\mid a(xy)=(ax)y+x(ay),\,\forall 
x,y\in\J\right\},$$
is $4$ whereas $\dim\operatorname{Jac}(\J_{13})=3$.
\item $\J_{7}\not\simeq\J_{16}$: the four-dimensional algebra 
$\F_{71}=\T_{4}\oplus\ka n_{4}$
is a subalgebra of $\J_{7}$ but is not  of $\J_{16}$.
\item $\J_{6}-\J_{11}-\J_{14}$: $\F_{63}\subseteq\J_{14}$ while $\F_{63}\not\subseteq\J_{6}$, adding $\dim H^{2}(\J_{11},\J_{11})=15$ while $\dim H^{2}(\J_{6},\J_{6})=\dim H^{2}(\J_{14},\J_{14})=14$ we obtain the result.
\item $\J_{24}^{\gamma}\not\simeq\J_{25}$: for any $\gamma\neq0$, algebra 
$\F_{69}$ is a subalgebra
of $\J_{24}^{\gamma}$ but is not a subalgebra
of $\J_{25}$. Since $\dim \Ann(\J_{24}^{0})=2$, $\,\J_{25}\not\simeq\J_{24}^{0}$.
\item $\J_{28}-\J_{24}^{0}-\J_{31}-\J_{35}-\J_{39}$: one compares the dimensions 
of both second cohomology
group and Jacobi space, $\dim H^{2}(\J_{35},\J_{35})=10$, 
$\dim H^{2}(\J_{31},\J_{31})=8$, $\dim H^{2}(\J_{28},\J_{28})=\dim 
H^{2}(\J_{39},\J_{39})=11$,
and $\dim H^{2}(\J_{24}^{0},\J_{24}^{0})=12$ , while 
$\dim\operatorname{Jac}(\J_{28})=3$ and $\dim\operatorname{Jac}(\J_{39})=4$.
\item $\J_{15}^{\alpha}-\J_{23}^{\beta}-\J_{26}^{\delta}-\J_{27}^{\varepsilon,\phi}$:
the $\nil(\J_{15}^{\alpha})=5$ for any $\alpha\in\ka$ while it is $4$ 
for the other algebras. Analogously, $\dim H^{2}(\J_{23}^{\beta},\J_{23}^{\beta})=8$
if $\beta\neq0$ and $\dim H^{2}(\J_{23}^{0},\J_{23}^{0})=9$, while
$\dim H^{2}(\J_{26}^{\delta},\J_{26}^{\delta})=\dim 
H^{2}(\J_{27}^{\varepsilon,\phi},\J_{27}^{\varepsilon,\phi})=7$
if $\delta\neq0$ and $\varepsilon\phi\neq1$. 

Next, suppose that
$\J_{27}^{\varepsilon,\phi}\simeq\J_{26}^{\delta}$ for some $\delta,\varepsilon,\phi\in\ka$
and suppose that this isomorphism is given by a matrix $A=(a_{ij})_{i,j=1}^{5}$. 
Then the coefficients $a_{ij}$ have to satisfy the following conditions: 
$a_{ij}=a_{i1}a_{i2}=a_{k5}=0$
for $i=1,2$, $j=3,4,5$,  and $a_{3k}=a_{1(k-2)}^{2}$, $a_{4k}=a_{2(k-2)}^{2}$,
$a_{5k}=+2a_{2(k-2)}(a_{3(k-2)}+a_{4(k-2)}\varepsilon)+2a_{1(k-2)}(a_{4(k-2)}+a_
{3(k-2)}\phi)=0$,
for $k=3,4$. Moreover,\small \begin{align*}
a_{21}a_{32}+a_{11}a_{42}+a_{21}a_{42}\varepsilon+a_{41}(a_{12}+a_{22}
\varepsilon)+a_{11}a_{32}\phi+a_{31}(a_{22}+a_{12}\phi) & =0\\
a_{11}a_{21}(a_{11}+a_{21})-a_{55}+a_{21}^{3}\varepsilon+a_{11}^{3}\phi & =0\\
-a_{55}+a_{21}(a_{12}^{2}+a_{22}^{2}\varepsilon)+a_{11}(a_{22}^{2}+a_{12}^{2}
\phi) & =0\\
a_{21}^{2}(a_{12}+a_{22}\varepsilon)+a_{11}^{2}(a_{22}+a_{12}\phi) & =0\\
a_{12}a_{22}(a_{12}+a_{22})-a_{55}\delta+a_{22}^{3}\varepsilon+a_{12}^{3}\phi & 
=0.
\end{align*}\normalsize
But the system allows only trivial solutions $a_{i5}=0$ for $1\leq i\leq 5$.
\item 
$\J_{26}^{0}-\J_{27}^{\varepsilon,\varepsilon^{-1}}-\J_{30}-\J_{32}-\J_{36}-\J_{
27}^{-1,-1}-\J_{41}$:
the nilpotency type of $\J_{32}$ and $\J_{36}$ is $(2,1,2)$ while
for all other algebras it is $(2,2,1)$ (for definition of nilpotency
type of Jordan algebra we refer to \cite{martin}).
Suppose that $\J_{36}\simeq\J_{32}$ and 
the isomorphism is given by the matrix $A=(a_{ij})_{i,j=1}^{5}$.
Then $a_{ij}$'s have to satisfy the following\
conditions: $a_{ij}=a_{3k}=2a_{i3}a_{33}=0$ for $i=1,2$, $j=2,4,5$, $k=4,5$, 
$a_{n2}=2a_{11}a_{(n-1)1}$, for $n=3,4$, and $a_{52}=2a_{31}a_{21}$,
$2a_{13}a_{23}=0$, $2a_{11}a_{21}a_{23}=a_{54}$, $2a_{11}a_{21}a_{13}=a_{44}$,
$a_{31}a_{23}+a_{21}a_{33}=a_{55}=2a_{11}a_{21}^{2}$, 
$a_{31}a_{13}+a_{11}a_{33}=a_{45}=2a_{11}^{2}a_{21}$,
$a_{21}a_{13}+a_{11}a_{23}=0$. This implies that $a_{i4}=0$ for $1\leq i\leq 5$. 

Comparing the dimensions of the cohomologies we rule out some cases:
$$\begin{array}{c}
\dim H^{2}(\J_{27}^{-1,-1},\J_{27}^{-1,-1})=\dim 
H^{2}(\J_{41},\J_{41})=8,\\ 
\dim H^{2}(\J_{27}^{\varepsilon,\varepsilon^{-1}},
\J_{27}^{\varepsilon,\varepsilon^{-1}})=\dim H^{2}(\J_{30},\J_{30})=7\\ 
\dim H^{2}(\J_{26}^{0},\J_{26}^{0})=10.
\end{array}
$$ 
On the other hand, $\dim\operatorname{Jac}(\J_{27}^{-1,-1})=4$ whereas    
$\dim\operatorname{Jac}(\J_{41})=3$.
It remains to check that  
$\J_{27}^{\varepsilon,\varepsilon^{-1}}\not\simeq\J_{30}$.
Suppose that $A=(a_{ij})_{i,j=1}^{5}$ which gives an isomorphism from 
$\J_{30}$ to $\J_{27}^{\varepsilon,\varepsilon^{-1}}$, then 
$a_{ij}=a_{k5}=0$ for $i=1,3$, $j=3,4,5$
 and $k=2,5$. Moreover, $a_{11}^{2}=a_{23}$, $2a_{21}a_{31}=a_{43}$,
$2a_{11}a_{31}+a_{31}^{2}=a_{53}$, $a_{11}a_{12}=0$, 
$a_{31}a_{22}+a_{21}a_{32}=0$,
$a_{11}a_{32}+a_{31}(a_{12}+a_{32})=0$, 
$a_{32}a_{24}\varepsilon^{-1}=a_{31}a_{23}\varepsilon=a_{45}=a_{31}a_{24}=a_{32}
a_{23}$,
$a_{12}^{2}=a_{24}$, $2a_{22}a_{32}=a_{44}$ and 
$2a_{12}a_{32}+a_{32}^{2}=a_{54}$.
All the solutions lead to $a_{i5}=0$ for $1\leq i\leq 5$, a contradiction. 
\end{enumerate}
\end{proof}

\section{Deformation, variety of Jordan algebras, constructing deformations\label{sec:conditions}}

Let $\V$ be a $n$-dimensional $\ka$-vector space with a fixed
basis $\{e_{1},e_{2},\cdots,e_{n}\}$. To endow $\V$ with a Jordan
$\ka$-algebra structure, $(\J,\cdot)$, it suffices to specify $n^{3}$
structure constants $c_{ij}^{k}\in\ka$, namely, 
$e_{i}\cdot e_{j}=\sum_{k=1}^{n}c_{ij}^{k}e_{k}$,  $i,j\in\{1,2,\dots,n\}$.
If a point $(c_{ij}^{k})\in\ka^{n^{3}}$ defines Jordan algebra its coordinates
satisfy certain polynomial equations which follow  from commutativity and Jordan identity. 
These equations cut out an algebraic variety $\Jor_{n}$ in $\ka^{n^{3}}$.
The points $(c_{ij}^{k})\in\Jor_{n}$ are in one-to-one correspondence with $n$-dimensional
$\ka$-algebras along with a particular choice of basis. A change of basis in
$\J$ is given by the general linear group $\G=\GL(\V)$ action on $\Jor_{n}$, namely: 
\begin{equation} 
g(\J,\cdot)\ \mapsto\ (\J,\cdot_{g}),\qquad x\cdot_{g}y=g(g^{-1}x\cdot g^{-1}y),\label{gl-action-1}
\end{equation}
for any $\J\in\Jor_{n}$, $g\in\G$ and $x,y\in\V$. The \textbf{$\G$-orbit}
of a Jordan algebra $\J$, that is, the set of all images of $\J$
under the action of $\G$, is denoted by $\J^{\G}$. The set of different
$\G$-orbits of this action is in one-to-one correspondence with the
set of isomorphism classes of $n$-dimensional Jordan algebras.
The Zariski topology of the affine space provides the following relation of 
orbits, namely, we say that $\J_{1}$ is a \textbf{deformation} of $\J_{2}$ or that $\J_{1}$
\textbf{dominates} $\J_{2}$, if the orbit $\J_{2}^{\G}$ is contained in the Zariski closure of
the orbit $\J_{1}^{\G}$. We denote it by $\J_{1}\to\J_{2}$. A Jordan algebra $\J$ is called \textbf{rigid}
if its $\G$-orbit is a Zariski-open set in $\Jor_{n}$.  Rigid algebra $\J$ does not admit non-trivial deformations, indeed if $\J_{1}\to \J$ then $\J_{1}^{\G}\cap\J^{\G}\neq\varnothing$ and therefore $\J_{1}\simeq\J$.

As we have already mentioned in the introduction the rigid algebras are of particular interest.
Recall that any affine variety could be decomposed into its irreducible components. 
When $\J\in\Jor_{n}$ is rigid, there exists an irreducible component
$\mathcal{C}$ such that $\J^{\G}\cap\mathcal{C}$ is a non-empty
open subspace in $\mathcal{C}$, and therefore the closure of $\J^{\G}$
coincides with $\mathcal{C}$. The converse is not true: in the sections 
\ref{sec:JorN5} we will show that there exists an irreducible component which does not contain 
a single open $\G$-orbit but rather an open subset which consists of 
infinite union of orbits of algebras.

%It follows from the next theorem proved
%by Flanigan for associative algebras in \cite{flanigan}. The proof
%for Jordan algebras is analogous.
%\begin{thm}
%\cite{flanigan} \label{thm:Flanigan}Let $\mathcal{C}$ be an irreducible
%component of $\Jor_{n}$, then: 
%
%$\mathcal{C}=\overline{\J^{\G}}$ for a rigid algebra $\J$ or, 
%
%$\mathcal{C}=\overline{\bigcup\limits _{\alpha\in I}\J_{\alpha}^{\G}}$
%where $I$ is an infinite set. 
%
%Moreover, in the second case the algebras $\J,\J',\dots$ in the union
%satisfy:
%\begin{enumerate}
%\item $\J_{ss}$ and $\J'_{ss}$ are isomorphic, where $\J=\J_{ss}+\Rad\J$
%is a direct sum (as vector space) of its radical $\Rad\J$ and a semisimple
%Wedderburn factor $\J_{SS}$ (see \cite[VII.6]{jacobson}); 
%\item $\dim\Rad\J=\dim\Rad\J'$ and this is minimal for algebras on $\mathcal{C}$; 
%\item \label{enu:nilindex}$\nil(\Rad\J)=\nil(\Rad\J')$ and this is maximal
%for algebras on $\mathcal{C}$; 
%\item the actions of $\J_{ss}$ and $\J'_{ss}$ on $\Rad\J$ and $\Rad\J'$,
%respectively, are equivalent;
%\item the orbits $\J^{\G},\J'^{\G},\cdots$ all have the same dimension. 
%\end{enumerate}
%\end{thm}
%As a consequence, the algebras $\J,\J',\dots$ must have non-isomorphic
%radicals.
%
Any $(c_{ij}^{k})\in\Jor_{n}$ which corresponds to a nilpotent Jordan algebra  
has to satisfy polynomial equation, thus the set
of nilpotent Jordan algebras in dimension $n$ possess a structure of algebraic variety embedded
in $\JorN_{n}\subseteq\Jor_{n}$. Further in this section we collect the properties of deformation
which will be used to describe $\JorN_5$. The first fact is straightforward 
but very useful.
\begin{lem}
\label{lem:directsum}\cite{kashubamartin} Let $\J_{i}$ and $\J'_{i}$
be in $\Jor_{n_{i}}$, for $i=1,2$. If $\J_{i}\to\J'_{i}$ , then
$\J_{1}\oplus\J_{2}\to\J'_{1}\oplus\J'_{2}$.\end{lem}
\begin{example}
\label{exa:directsum} From the description of $\JorN_4$ in \cite{ancocheabermudes} it follows that
$\mathcal{\F}_{62}\to\mathcal{\F}_{63}\to\mathcal{\F}_{64}$ and $\mathcal{\F}_{62}\to\mathcal{\F}_{65}$, consequently
$\J_{4}\to\J_{3}\to\J_{2}$ and $\J_{4}\to\J_{1}$.
\end{example}
The other way to deform one algebra into another is to build a curve as
specified in Lemma \ref{lem:curva}.
\begin{lem}
\label{lem:curva}\cite{mazzolacocycles}Suppose that there exists a curve
$\gamma$ in $\Jor_{n}$ which generically lies in a subvariety $\mathcal{U}$
and which cuts $\J^{\G}$ in special point then $\J\in\mathcal{\overline{U}}$.\end{lem}
\begin{example}
\label{exa:curva}Let $\gamma(t)=\left\{ e_{1}^{t},\!e_{2}^{t},\!e_{3}^{t},\!e_{4}^{t},\!e_{5}^{t}\right\} $
be a change of basis of $\J_{21}$ given by: $e_{1}^{t}=\sqrt{2}t^{2}n_{1}-\sqrt{2}t^{2}n_{2}+2t^{2}n_{3}-t^{2}n_{4}$,
$e_{2}^{t}=2\sqrt{2}t^{4}n_{2}+4t^{4}n_{4}$, $e_{3}^{t}=\sqrt{2}t^{3}(n_{1}+n_{2})$,
$e_{4}^{t}=4t^{6}(n_{4}+n_{5})$ and $e_{5}^{t}=8t^{8}n_{5}$. Thus 
$\gamma$ defines a curve in $\JorN_{5}$ which for any $t\neq 0$ belongs to 
$\J_{21}^{\G}$ (observe that $\det\gamma(t)=-2^{8}t^{23}$).
However, $e_{2}^{t}e_{3}^{t}=te_{4}^{t}$ and for $t\to 0$ we get the 
multiplication table of $\J_{18}$.
Thus, from Lemma \ref{lem:curva}, it follows that 
$\J_{18}\in\overline{\J_{21}^{\G}}$,
i.e., $\J_{21}\to\J_{18}$.
\end{example}
The non-existence of deformation for a given pair of algebras $\J_{1},\J_{2}\in\Jor_{n}$
will follow from the violation of one of the conditions below. 
\begin{thm}
\label{thm:conditions}If $\J_{1}\rightarrow\J_{2}$ then:
\begin{enumerate}
\item \label{enu:Aut}$\dim\Aut(\J_{1})<\dim\Aut(\J_{2})$; 
\item \label{enu:Ann}$\dim\Ann(\J_{1})\leq\dim\Ann(\J_{2})$; 
\item \label{enu:Pot}$\dim(\J_{1}^{m})\geq\dim(\J_{2}^{m})$, for any positive
integer $m$;
\item \label{enu:ZAssoc}$\dim\Z(\J_{1})\leq\dim\Z(\J_{2})$.
\end{enumerate}
\end{thm}
\begin{proof}
For the proof of \eqref{enu:Aut} to \eqref{enu:Pot} we refer to
\cite{kashubamartin}. To see \eqref{enu:ZAssoc}, it suffices
to prove that the sets $\{\J\in\Jor_{n}\mid\dim\Z(\J)\geq s\}$ are
Zariski-closed in $\Jor_{n}$ for any $s\in\mathbb{N}$. Let ${e_{1},e_{2},\dots,e_{n}}$
be a basis for $\J\in\Jor_{n}$ and $a=\sum_{i=1}^{n}\alpha_{i}e_{i}\in\Z(\J)$.
Substituting $a$ in each of the $3n^{2}$ equations $(a,e_{i},e_{j})=0$,
$(e_{i},a,e_{j})=0$ and $(e_{i},e_{j},a)=0$ for $i,j=1,\dots,n$,
we obtain a system with $3n^{3}$ equations having $\alpha_{i}$ as
unknowns. Then $\dim\Z(\J)=n-\operatorname{rank}(P_{3n^{3},n})$,
where $P_{3n^{3},n}$ is the matrix of the system, thus the condition
$\dim\Z(\J)\geq s$ is equivalent to the fact that all $(n-s+1)$-minors
of $P_{3n^{3},n}$ vanish, which is defined by a finite number
of polynomial identities in the structure constants.
\end{proof}
As direct consequences of \eqref{enu:Pot} and \eqref{enu:ZAssoc},
respectively, we have:
\begin{cor}
\label{cor:assoc}\mbox{}
\begin{enumerate}
\item If $\J_{1}$ and $\J_{2}$ are nilpotent Jordan algebras such as $\J_{1}\to\J_{2}$
then $\nil(\J_{1})\geq\nil(\J_{2})$.
\item Any deformation of a non-associative algebra is a non-associative
algebra. 
\end{enumerate}
\end{cor}

\section{Deformations of infinite families of algebras\label{sec:examples}}

Whenever the class of algebras consists of only finitely many non-isomor\-phic 
algebras, the irreducible components are completely determined by rigid 
algebras, 
see $\Jor_n$, $n\leq 4$. When we have infinite number of non-isomorphic 
algebras it is not necessarily that any irreducible component contains 
an open orbit. In \cite{flanigan}, Flanigan showed that there exists an infinite 
family of nilpotent algebras of dimension three such that the closure of a 
union 
of their orbits is an irreducible component in $\AssocN_3$, while it carries no 
open orbit. In the next section, we will show that such irreducible component 
exists in $\JorN_5$.
In general, the behaviour of infinite families under deformation is of 
particular interest and in what follows we will look into examples of 
deformations in $\JorN_5$ which involve families.
First, to establish the notation denote by
$$\mathcal{N}_{i}^{\#}=\bigcup\limits_{\alpha\in T}\left(\J_i^{\alpha}\right)^{\G},$$
the union of the orbits of algebras in the family $\J^{\alpha}_i$, $\alpha\in 
T$ where $T$ is some subset of $\ka$.

\subsection{Infinite family is dominated by a single 
algebra}\label{exa:22->15}The infinite family $\J_{15}^{\alpha},\ 
\alpha\in\ka$, 
is dominated by algebra $\J_{22}$. To see that, for any $\alpha\in\ka$ consider a change of basis of $\J_{22}$
parametrized by $\gamma_{\alpha}(t)$: $e_{1}^{t}=t^{2}n_{1}-\frac{\alpha}{3}n_{2}+tn_{3}$,
$e_{2}^{t}=t^{4}n_{2}+(1-\frac{2\alpha}{3})t^{2}n_{4}+\frac{\alpha^{2}}{9}n_{5}
$,
$e_{3}^{t}=t^{6}n_{4}-(-1+\alpha)t^{4}n_{5}$, $e_{4}^{t}=t^{4}n_{2}-t^{5}n_{3}+\frac{\alpha t^{2}}{3}n_{4}$
and $e_{5}^{t}=t^{8}n_{5}$. Observe that $\det(\gamma_{\alpha})=t^{25}$ thus, 
for any $t\neq0$, the curve $\gamma_{\alpha}(t)$
belongs to $\J_{22}^{\G}$. Since $\left(e_{4}^{t}\right)^{2}=t^{4}e_{3}^{t}+\alpha e_{5}^{t}$ it follows that $\gamma_{\alpha}(0)$ gives the multiplication table 
of $\J_{15}^{\alpha}$. Therefore 
$\mathcal{N}_{15}^{\#}\subset\overline{\J_{22}^{\G}}$.

\subsection{Infinite family is a deformation of a single 
algebra\label{exa:deformation-of-family}}

The deformation from the family $\J_{23}^{\beta},\ \beta\in\ka$ to $\J_{41}$ is quite remarkable, since no single member of
$\J_{23}^{\beta}$ can dominate $\J_{41}$. Indeed $\dim\Aut(\J_{41})=\dim\Aut(\J_{23}^{\beta})$
for all $\beta\in\ka$. Consider
the ``variable'' change of basis of $\J_{23}^{\beta(t)}$, where
$\beta(t)=\frac{1}{t}$, namely: $e_{1}^{t}=tn_{1}-\frac{1}{2}n_{2}$,
$e_{2}^{t}=-\frac{1}{2}n_{2}$, $e_{3}^{t}=-\frac{t}{2}n_{3}$, 
$e_{4}^{t}=\frac{t}{4}n_{5}$
and $e_{5}^{t}=t^{2}n_{3}-tn_{4}$. So the curve $\gamma(t)$ lies transversely 
to the orbits of $\J_{23}^{\beta}$, meaning that, for any $t\neq 0$, 
$\gamma(t)\subset\mathcal{N}_{23}^{\#}$
and cuts $\J_{41}^{\G}$ in $t=0$. Thus $\J_{41}\in\overline{\mathcal{N}_{23}^{\#}}$. We
will abuse the notation and will denote this fact by $\J_{23}^{\beta}\to\J_{41}$. 

Observe that in the same way we obtain that any algebra $\J^{\alpha}_{k}$ of a family 
can be viewed as a deformation of $\mathcal{N}_{k}^{\#}$. We use the basis of $\J^{\alpha}_{k}$ from Table 
\ref{tab:5dnilponaoassoc}, then for any fixed $\alpha_0$ $\gamma(\alpha)$ defines the multiplication 
in $\J^{\alpha_0}_{k}$ while for all other values of $\alpha$ algebra belongs to $\mathcal{N}_{k}^{\#}$.

%As a consequence, the changes of basis of $\J_{24}^{t}$, $\J_{24}^{t+1}$,
%$\J_{24}^{t-\frac{1}{2}}$, $\J_{26}^{t}$, $\J_{27}^{\varepsilon-t,\varepsilon^{-1}}$
%and $\J_{27}^{t-1,(t-1)^{-1}}$given by the identity, provide us 
%$\J_{24}^{\gamma}\to\J_{24}^{0},\epsilon_{4},\J_{24}^{-\frac{1}{2}}$,
%$\J_{26}^{\delta}\to\J_{26}^{0}$ and 
%$\J_{27}^{\varepsilon,\phi}\to\J_{27}^{\varepsilon,\varepsilon^{-1}}\to\J_{27}^
%{ (-1,-1)}$. 

In a rather  different way, we obtain the dominance of
$\J_{39}$ by the family $\J_{15}^{\alpha},\ \alpha\in\ka$. In this
case $\J_{15}^{1}$ (a single member of $\J_{15}^{\alpha},\ \alpha\in\ka$)
dominates $\J_{39}$. Consider the curve: $e_{1}^{t}=tn_{1}$, $e_{2}^{t}=n_{2}+(-1+t)n_{4}$,
$e_{3}^{t}=tn_{3}$, $e_{4}^{t}=t^{2}n_{5}$ and $e_{5}^{t}=t^{2}n_{2}$.
For any $t\neq0,1$, the curve lies in $\left(\J_{15}^{1}\right)^{\G}$
and cuts $\J_{39}^{\G}$ in $t=0$ and $\epsilon_{5}^{\G}$ in $t=1$,
respectively. Then we have 
$$\J_{39}\in\overline{\left(\J_{15}^{1}\right)^{\G}}\subseteq\bigcup_{\alpha\in\ka}\overline{\left(\J_{15}^{\alpha}\right)^{\G}}\subseteq\overline{\mathcal{N}_{15}^{\#}}$$
and, therefore, $\J_{15}^{\alpha}\to\J_{39}$.

\subsection{Deformations between families}\label{deformation-between-families}
The family $\J_{27}^{\varepsilon,\phi}$, $\varepsilon,\phi\in\ka\mbox{ and \ensuremath{\varepsilon\phi\neq1}}$,
dominates the family $\J_{26}^{\delta}$, $\delta\in\ka^{*}$. Again no single 
member of  $\J_{26}^{\delta}$ deforms into any single member of  $\J_{27}^{\varepsilon,\phi}$
since $\dim\Aut(\J_{26}^{\delta})=\dim\Aut(\J_{27}^{\varepsilon,\phi})$,
for all $\delta,\varepsilon,\phi\in\ka$.  For any fixed $\delta\in\ka^{*}$,
we construct a curve $\gamma_{\delta}(t)$ which lie transversely
to the orbits of $\J_{27}^{\varepsilon,\phi}$. Consider the ``variable''
change of basis of $\J_{27}^{\varepsilon(t),\phi(t)}$, where $(\varepsilon(t),\phi(t))=(t\delta,\frac{1}{t^{2}})$,
defined as: $e_{1}^{t}=t\delta n_{1}$, $e_{2}^{t}=\delta n_{2}$,
$e_{3}^{t}=t^{2}\delta^{2}n_{3}$, $e_{4}^{t}=\delta^{2}n_{4}$
and $e_{5}^{t}=t\delta^{3}n_{5}$. For any $t\neq0$, the curve
lies in $\mathcal{N}_{27}^{\#}$ and, since $e_{2}^{t}e_{3}^{t}=te_{5}^{t}$,
cuts $\left(\J_{26}^{\delta}\right)^{\G}$ in $t=0$. Thus 
$\J_{26}^{\delta}\in\overline{\mathcal{N}_{27}^{\#}}$,
for any $\delta\in\ka^{*}$.

Differently, we obtain the dominance of the family
$\J_{24}^{\gamma}$, $\gamma\in\ka^{*}-\{1,-\frac{1}{2}\}$ by the
family $\J_{23}^{\beta},\ \beta\in\ka$. In this case, fixing $\gamma_{0}\in\ka^{*}-\{1,-\frac{1}{2}\}$,
the algebra $\J_{24}^{\gamma_{0}}$ deforms into $\J_{23}^{\beta_{0}}$, where $\beta_{0}=\frac{(1-2\gamma_{0})i}{\sqrt{\gamma_{0}}}$. 
Define the curve: $e_{1}^{t}=-\frac{i}{\sqrt{\gamma_{0}}}n_{1}+n_{2}$,
$e_{2}^{t}=tn_{2}$, $e_{3}^{t}=-\frac{1}{\gamma_{0}}n_{3}-\frac{2i}{\sqrt{\gamma_{0}}}n_{4}$,
$e_{4}^{t}=-\frac{it}{\sqrt{\gamma_{0}}}n_{4}$ and $e_{5}^{t}=-\frac{it}{\sqrt{\gamma_{0}^{3}}}n_{5}$.
For any $t\neq0$, the curve lies in the $\G$-orbit of $\J_{23}^{\beta_{0}}$
and, since $e_{2}^{t}e_{4}^{t}=t\gamma_{0}e_{5}^{t}$, cuts $\left(\J_{24}^{\gamma_{0}}\right)^{\G}$
in $t=0$. Therefore $$\J_{24}^{\gamma_{0}}\in\overline{\left(\J_{23}^{\beta_{0}}\right)^{\G}}
\subseteq\overline{\mathcal{N}_{23}^{\#}}$$
consequently $\J_{23}^{\beta}\to\J_{24}^{\gamma}$, for all $\gamma\in\ka^{*}-\{1,-\frac{1}{2}\}$.

\section{The algebraic variety $\protect\JorN_{5}$\label{sec:JorN5}}

 %The varieties $\JorN_{3}$ and $\operatorname{AssoCN}_{3}$, of associative,
%commutative and nilpotent algebras of dimension three, coincide. Then
%it follows from \cite{mazzolacocycles} that $\JorN_{3}$ is an irreducible
%variety given by the Zariski closure of the $\G$-orbit of the algebra
%$\T_{3}$. Furthermore, we have the following deformations 
%$\T_{3}\to\T_{4}\to\B_{3}\oplus\ka n\to\ka n\oplus\ka n\oplus\ka n$
%(see \cite{irynashesta}).  

In this section, we will describe the algebraic variety of five-dimensional 
Jordan algebras $\JorN_{5}$.
First we show that any structure can be deformed either into one of the given 
four algebras or into  the algebras in family $\J_{27}^{\varepsilon,\phi}$.
The main tool of the proof is Lemma \ref{lem:curva}, to define a curve 
$\gamma(t)$ we use notations and basis of Table \ref{tab:5dnilponaoassoc}.

\begin{lem}
\label{lem:notrigid}Every structure $\J\in\JorN_{5}$ belongs to
$\overline{\epsilon_{1}^{\G}}$, $\overline{\mathcal{\J}_{21}^{\G}}$,
$\overline{\mathcal{\J}_{22}^{\G}}$, $\overline{\mathcal{\J}_{40}^{\G}}$
or $\overline{\mathcal{N}_{27}^{\#}}$. \end{lem}
\begin{proof}

All associative structures, namely $\epsilon_{i}$ for $i=1,\dots,25$,
belong to $\overline{\epsilon_{1}^{\G}}$. In fact, G. Mazzola in
\cite{mazzolacocycles} showed that the subvariety of associative algebras
is an irreducible subvariety in $\JorN_{5}$ and  that the orbit 
$\epsilon_{1}^{\G}$
is open, thus its closure contains all associative structure. From Corollary
\ref{cor:assoc} it follows that there is no non-associative structure which deforms
into $\epsilon_{1}$.
 
All algebras $\J_{i}$ for $i=1,\dots,20$ except for the family 
$\J^{\alpha}_{15}$, are dominated
by $\J_{21}$. Example \ref{exa:curva} gives $\J_{21}\to\J_{18}$.
For the deformation $\J_{18}\rightarrow\J_{17}$, consider a family of algebras with the following basis
\begin{equation}\label{J18-example}
\begin{array}{l}
e_{1}^{t}=-t^{2}n_{1}+\frac{1}{2}n_{2}-tn_{3}+\frac{1}{8t^{2}}n_{4},\\
e_{2}^{t}=-t^{8}n_{2}-\frac{3t^{6}}{2}n_{4},\\ 
e_{3}^{t}=t^{4}n_{2}-t^{5}n_{3}+\frac{t^{2}}{2}n_{4},\\
e_{4}^{t}=t^{10}n_{4}+t^{8}n_{5},\\
e_{5}^{t}=-t^{12}n_{5}.
\end{array}
\end{equation}
Here $\{n_{1},\dots,n_{5}\}$ is the basis of $\J_{18}$ given in Table \ref{tab:5dnilponaoassoc}, $t\in\ka$ and therefore
this family can be viewed as a curve $\gamma(t)\in \JorN_{5}$. For any fixed 
$t\neq 0$ the change of basis \eqref{J18-example}
determines algebras isomorphic to $\J_{18}$, while for $t=0$ one obtains $\J_{17}$ (indeed, $\left(e_{2}^{t}\right)^{2}=-t^{4}e_{5}^{t}$).
Analogously, the change of basis $e_{1}^{t}=n_{1}$, $e_{2}^{t}=t^{2}n_{2}$,
$e_{3}^{t}=tn_{3}$, $e_{4}^{t}=t^{2}n_{4}$ and $e_{5}^{t}=t^{2}n_{5}$
of $\J_{17}$ does the trick for $\J_{17}\rightarrow\J_{16}$. 
Applying Lemma \ref{lem:curva} to given transformations
we obtain the following deformations:
$$
\begin{array}{c|c|c|c|c}
\J_{21}\rightarrow\J_{20} & \J_{20}\rightarrow\J_{19} & \J_{19}\rightarrow\J_{12}& \J_{12}\to\J_{11} & \J_{20}\to\J_{8}\\
e_{1}^{t}=\frac{1}{t^{3}}n_{1} & e_{1}^{t}=t^{2}n_{1} & e_{1}^{t}=n_{1}-\frac{1}{2}n_{2}+n_{3} & e_{1}^{t}=n_{1} & e_{1}^{t}=tn_{1}\\
e_{2}^{t}=\frac{1}{t}n_{2} & e_{2}^{t}=n_{2}& e_{2}^{t}=tn_{3}+n_{4} & e_{2}^{t}=t^{2}n_{2} & e_{2}^{t}=n_{2}\\
e_{3}^{t}=\frac{1}{t^{2}}n_{3} & e_{3}^{t}=tn_{3} & e_{3}^{t}=tn_{4}+n_{5}& e_{3}^{t}=t^{2}n_{3} & e_{3}^{t}=tn_{4}\\
e_{4}^{t}=\frac{1}{t^{4}}n_{4} &  e_{4}^{t}=t^{2}n_{4} & e_{4}^{t}=it(n_{2}-n_{3})& e_{4}^{t}=tn_{4} & e_{4}^{t}=tn_{3}+t^{2}n_{4} \\
e_{5}^{t}=\frac{1}{t^{6}}n_{5} & e_{5}^{t}=t^{3}n_{5} & e_{5}^{t}=tn_{5}& e_{5}^{t}=t^{2}n_{5} & e_{5}^{t}=t^{2}n_{5} \\
\end{array}
$$

To obtain $\J_{21}\to\J_{9}$ consider the following change of basis
$e_{1}^{t}=-\frac{1}{\sqrt{2}t}n_{1}$, $e_{2}^{t}=\frac{1}{\sqrt{2}t}n_{2}$,
$e_{3}^{t}=-\frac{1}{2t^{2}}n_{4}$, $e_{4}^{t}=-\frac{1}{\sqrt{2}}n_{1}+\frac{1}{\sqrt{2}}n_{2}-n_{3}+\left(\frac{1}{2}-\frac{1}{4t}\right)n_{4}$
and $e_{5}^{t}=\frac{1}{2t^{2}}n_{5}$ of $\J_{21}$. 
Deformations $\J_{4}\to\J_{3}\to\J_{2}$ and $\J_{4}\to\J_{1}$ follow from Example \ref{exa:directsum}. 
Suppose $\{n_{1},\dots,n_{5}\}$ is the basis of $\J_{9}$, then $e_{1}^{t}=tn_{1}-\frac{1}{2}n_{3}+t^{2}n_{4}$,
$e_{2}^{t}=tn_{2}-\frac{1}{2}n_{3}+t^{2}n_{4}$, $e_{3}^{t}=t^{2}(n_{3}-n_{5})$,
$e_{4}^{t}=\frac{t}{2}n_{3}+t^{3}n_{4}$ and $e_{5}^{t}=t^{4}n_{5}$ for any
$t\neq0$ defines algebras which are isomorphic to $\J_{9}$. Meanwhile 
$e_{3}^{t}e_{4}^{t}=te_{5}^{t}$ thus, when $t$ tends to zero, we get the 
structure constants of
algebra $\J_{10}$. The basis $e_{1}^{t}=n_{1}-t^{2}n_{2}+\sqrt{2}tn_{4}$,
$e_{2}^{t}=-t^{2}n_{1}+n_{3}-\sqrt{2}tn_{4}$, $e_{3}^{t}=t^{4}n_{3}-4t^{2}n_{5}$,
$e_{4}^{t}=\frac{t^{4}}{2}n_{3}-\frac{t^{5}}{2\sqrt{2}}n_{4}$ and
$e_{5}^{t}=-t^{6}n_{5}$ of $\J_{10}$ does the trick for $\J_{10}\rightarrow\J_{13}$. To complete the description of orbits inside $\J_{21}^{\G}$
use Lemma \ref{lem:curva} for the following cases:
$$
\begin{array}{c|c|c|c|c}
\J_{8}\to\J_{7}& \J_{7}\to\J_{4}& \J_{13}\to\J_{14} &\J_{10}\to\J_{5}& \J_{5}\to\J_{6}\\
e_{1}^{t}=n_{2}& e_{1}^{t}=n_{1}+n_{4}& e_{1}^{t}=n_{2} & e_{1}^{t}=\frac{it}{\sqrt{2}}(n_{1}-n_{2})& e_{1}^{t}=n_{1}+n_{4}\\
e_{2}^{t}=tn_{1}& e_{2}^{t}=n_{2}+n_{4}& e_{2}^{t}=t(-n_{1}+n_{2})& e_{2}^{t}=t^{2}n_{3}& e_{2}^{t}=n_{2}+n_{5} \\
e_{3}^{t}=tn_{3}& e_{3}^{t}=n_{3}& e_{3}^{t}=-tn_{3}& e_{3}^{t}=\frac{t^{2}}{2}(n_{1}+n_{2})& e_{3}^{t}=tn_{3}\\
e_{4}^{t}=n_{4}& e_{4}^{t}=n_{5}& e_{4}^{t}=n_{4}& e_{4}^{t}=t^{2}n_{4} & e_{4}^{t}=tn_{4}\\
e_{5}^{t}=tn_{5}& e_{5}^{t}=tn_{1}& e_{5}^{t}=-tn_{5}& e_{5}^{t}=t^{4}n_{5}& e_{5}^{t}=tn_{5}\\
\end{array}
$$

Any algebra $\J_{28}$, $\J_{30}\,$-$\J_{34}$, $\J_{39}$ as well 
as the family $\J_{15}^{\alpha}$, $\alpha\in\ka$ deforms into algebra 
$\J_{22}$. 
Deformations  $\J_{22}\to\J_{15}^{\alpha}\to\J_{39}$ follow from Sections \ref{exa:22->15} and \ref{exa:deformation-of-family}. 
To show that $\J_{39}\to\J_{34}$ we use again Lemma \ref{lem:curva} applying it to the curve $\gamma(t):$ $e_{1}^{t}=n_{1}$, $e_{2}^{t}=tn_{2}$,
$e_{3}^{t}=tn_{3}$, $e_{4}^{t}=tn_{4}$ and $e_{5}^{t}=n_{5}$ of
$\J_{39}$. The deformation $\J_{22}\to\J_{33}$ is given by the change
of basis $e_{1}^{t}=t^{2}n_{1}+n_{3}$, $e_{2}^{t}=t^{4}n_{2}+n_{4}$,
$e_{3}^{t}=tn_{2}-t^{3}n_{3}$, $e_{4}^{t}=t^{5}n_{5}$ and $e_{5}^{t}=t^{6}n_{4}+t^{2}n_{5}$
of $\J_{22}$. The curve $\gamma(t):$
$$ 
\begin{array}{l}
e_{1}^{t}=t^{2}n_{1}+n_{2}+tn_{3},\\ 
e_{2}^{t}=t^{4}n_{2}+2tn_{4}+3t^{2}n_{5},\\
e_{3}^{t}=\frac{t^{4}}{3}n_{2}+\frac{2t^{5}}{3}n_{3},\\ 
e_{4}^{t}=\frac{2t^{9}}{3}n_{4},\\
e_{5}^{t}=t^{5}n_{4}+t^{6}n_{5}
\end{array}
$$
belongs to $\J_{33}^{\G}$ for any
$t\neq0$ and intersects $\J_{32}$ when $t=0$, thus $\J_{33}\to\J_{32}$. The new basis 
$e_{1}^{t}=n_{1}$, $e_{2}^{t}=n_{2}$, $e_{3}^{t}=t(n_{2}-n_{3})$,
$e_{4}^{t}=-tn_{4}$ and $e_{5}^{t}=n_{5}$ of $\J_{32}$ for any $t\neq 0$ defines algebra isomorphic to $\J_{32}$.
Observe that $\left(e_{3}^{t}\right)^{2}=2te_{4}^{t}$, thus we get the structure
$\J_{31}$ when $t$ tends to zero, i.e. $\J_{32}\to\J_{31}$. The
changes of basis $e_{1}^{t}=t^{2}n_{1}+\frac{1}{9}n_{2}+\frac{t}{3}n_{3}$,
$e_{2}^{t}=t^{4}n_{2}+\frac{2t}{27}n_{4}+\frac{t^{2}}{3}n_{5}$, $e_{3}^{t}=\frac{1}{12}n_{2}+\frac{t}{2}n_{3}$,
$e_{4}^{t}=\frac{t^{5}}{2}n_{4}$ and $e_{5}^{t}=\frac{t}{12}n_{4}+\frac{t^{2}}{4}n_{5}$
of $\J_{33}$ and $e_{1}^{t}=t(n_{1}+n_{2})$, $e_{2}^{t}=t^{2}n_{2}$,
$e_{3}^{t}=\frac{1}{2}n_{2}+n_{3}$, $e_{4}^{t}=t^{2}n_{4}$ and $e_{5}^{t}=n_{4}+n_{5}$
of $\J_{30}$ give $\J_{33}\to\J_{30}$ and $\J_{30}\to\J_{28}$, respectively.

Next, $\J_{40}\to\J_{38}\to\J_{36}$. Indeed, we apply Lemma \ref{lem:curva}
to the change of basis $e_{1}^{t}=\frac{1}{t^{2}}n_{1}$, $e_{2}^{t}=\frac{1}{t}n_{2}$,
$e_{3}^{t}=\frac{1}{t^{3}}n_{3}$, $e_{4}^{t}=\frac{1}{t^{5}}n_{4}$
and $e_{5}^{t}=\frac{1}{t^{4}}n_{5}$  of $\J_{40}$ and  $e_{1}^{t}=n_{2}$, $e_{2}^{t}=tn_{1}$,
$e_{3}^{t}=tn_{3}$, $e_{4}^{t}=tn_{5}$ and $e_{5}^{t}=t^{2}n_{4}$,
of  $\J_{38}$, respectively.

Finally, the orbits of remaining structures, i.e. families $\J_{23}^{\beta}$, $\J_{24}^{\gamma}$, $\J_{26}^{\delta}$ as well as algebras $\J_{25}$, $\J_{35}$, $\J_{41}$,
belong to the closure of $\mathcal{N}_{27}^{\#}$. 
The deformations $\J_{27}^{\varepsilon,\phi}\to\J_{26}^{\delta}$,
$\J_{23}^{\beta}\to\J_{41}$ and  $\J_{23}^{\beta}\to\J_{24}^{\gamma}$
follow from Sections 
\ref{exa:deformation-of-family} and \ref{deformation-between-families}.

Further, we show that $\J_{27}^{\varepsilon,\phi}\to \J_{23}^{\beta_0}$ for any 
fix $\beta_0\in\ka$. 
Consider the change of basis of $\J_{27}^{\varepsilon(t),\phi(t)}$: 
$$
\begin{array}{l}
e_{1}^{t}=(-1-t\beta_{0})n_{1}-n_{2},\\
e_{2}^{t}=tn_{2},\\
e_{3}^{t}=(1+t\beta_{0})^{2}n_{3}+n_{4},\\
e_{4}^{t}=-tn_{4},\\
e_{5}^{t}=t^{2}(1+t\beta_{0})n_{5},
\end{array}
\qquad (\varepsilon(t),\phi(t))=\left(-1-t\beta_{0},\frac{-1-t^{2}-t\beta_{0}}{(1+t\beta_{0})^{2}}\right).
$$
Note that $\varepsilon(t)\phi(t)=1+\frac{t^{2}}{1+t\beta_{0}}$
then, while for any $t\neq0$ this basis defines structure in $\mathcal{N}_{27}^{\#}$, 
for $t=0$ it gives $\J_{23}^{\beta_{0}}$. 

Applying Lemma \ref{lem:curva} to the basis transformation
$e_{1}^{t}=n_{2}$, $e_{2}^{t}=tn_{1}-tn_{2}$,
$e_{3}^{t}=tn_{3}$, $e_{4}^{t}=tn_{4}$ and 
$e_{5}^{t}=-2t^{2}n_{3}+t^{2}n_{5}$
of $\J_{41}$ we obtain $\J_{41}\to\J_{35}$.

To show that $\J^{\gamma}_{24}\to\J_{25}$, define a curve $\phi(t)\subset \JorN_5$,
as been the set of structure constants of algebras with the basis 
$e_{1}^{t}=n_{1}+n_{2}$, $e_{2}^{t}=-2tn_{2}$, $e_{3}^{t}=n_{3}+2n_{4}$,
$e_{4}^{t}=-2tn_{4}$ and $e_{5}^{t}=tn_{5}$, where $\{n_1,\dots,n_5\}$ is a basis of $\J_{24}^{\gamma(t)}$, $\gamma(t)=\frac{t-1}{2}$.
Analogously to Section \ref{exa:deformation-of-family}, the curve lies transversely
to the orbits of $\J_{24}^{\gamma}$. For any $t\neq0$,
$\phi(t)\subset\mathcal{N}_{24}^{\#}$, while
$\phi(0)$ corresponds to the structure constants of  $\J_{25}$. \end{proof}
\begin{lem}
\label{lem:27componente}The infinite union $\mathcal{N}_{27}^{\#}$
is not contained in the Zariski closure of the orbits of any algebra
in $\JorN_{5}$.\end{lem}
\begin{proof}
%As a consequence of Lemma \ref{lem:notrigid} it is sufficient to
%find a pair $(\varepsilon_{0},\phi_{0})\in\ka^{2}$ with $\varepsilon_{0}\phi_{0}\neq1$
%such as the $\G$-orbit of $\mathcal{\J}_{27}^{\varepsilon_{0},\phi_{0}}$
%is not contained in $\overline{\epsilon_{1}^{\G}}$, $\overline{\mathcal{\J}_{21}^{\G}}$,
%$\overline{\mathcal{\J}_{22}^{\G}}$ neither 
%in $\overline{\mathcal{\J}_{40}^{\G}}$.
By Corollary \ref{cor:assoc} the non-associative algebra 
$\mathcal{\J}_{27}^{\varepsilon,\phi}$
can not be deformed into the associative algebra $\epsilon_{1}$ for any 
$\varepsilon,\phi\in\ka$. 
Observe that for any $\varepsilon,\phi\in\ka$ with $\varepsilon\phi\neq1$,  
$\dim\left(\J_{27}^{\varepsilon,\phi}\right)^{2}=3$ and 
$\dim\Ann(\J_{27}^{\varepsilon,\phi})=1$. Applying Theorem 
\ref{thm:conditions}
item \ref{enu:Pot} and item \ref{enu:Ann} we obtain that 
$\J_{21}\nrightarrow\mathcal{\J}_{27}^{\varepsilon,\phi}$ and 
$\J_{40}\nrightarrow\mathcal{\J}_{27}^{\varepsilon,\phi}$, 
when $\varepsilon\phi\neq1$, respectively. 

To see that $\mathcal{\J}_{22}\nrightarrow\J_{27}^{\varepsilon,\phi}$, first
observe that $\J_{22}$ does not deform into the other three algebras from Lemma 
\ref{lem:notrigid}. Indeed, comparing the dimensions of the  automorphism group 
we deduce that $\epsilon_{1}\nrightarrow\J_{22}$
and $\J_{40}\nrightarrow\J_{22}$, while the dimensions of the second power provide that $\J_{22}^{\G}\nsubseteq\overline{\mathcal{\J}_{21}^{\G}}$.
Suppose that $\mathcal{N}_{27}^{\#}\subseteq\overline{\mathcal{\J}_{22}^{\G}}$
then $\J_{22}$ is rigid and $\mathcal{C}=\overline{\mathcal{\J}_{22}^{\G}}$
is a irreducible component of $\JorN_{5}$. Since
$\overline{\mathcal{N}_{27}^{\#}}$ is contained in the irreducible component
$\mathcal{C}$ and both has the same dimension $\dim\overline{\mathcal{N}_{27}^{\#}}=21=\dim\mathcal{C}$,
$\overline{\mathcal{N}_{27}^{\#}}=\mathcal{C}$. So, $\mathcal{N}_{27}^{\#}$ is 
a dense subset of $\mathcal{C}$ and $\J_{22}^{\G}$ is open in $\mathcal{C}$ 
then there exists $\J\in \mathcal{N}_{27}^{\#}$ such as $\J\in \J_{22}^{\G}$, 
i.e., $\J\simeq\J_{22}$, a contradiction by Theorem \ref{thm:classification}.

\end{proof}
\begin{cor} $\overline{\mathcal{N}_{27}^{\#}}$ determines a component of
the variety $\JorN_{5}$. In particular, we proved that a finite-dimensional
Jordan algebras not necessarily deforms into rigid one.
\end{cor}
\begin{proof} From Lemma \ref{lem:27componente} it follows that we only have to show that all algebras in 
$\overline{\mathcal{N}_{27}^{\#}}$ belong to the same component. In fact 
$\overline{\mathcal{N}_{27}^{\#}}$ is a product of two irreducible varieties  
$\left(\J_{27}^{\varepsilon,\phi}\right)^{\G}$ and $\ka^2=\{(\varepsilon,\phi)\,|\, \varepsilon,\phi\in \ka\}$ and is therefore irreducible. 
\end{proof}
\begin{lem}
\label{lem:rigid} Algebras $\epsilon_{1}$, $\J_{21}$, $\J_{22}$
and $\J_{40}$ are rigid.\end{lem}
\begin{proof}
By a similar argument as in proof of Lemma \ref{lem:27componente} we can show  that
${\mathcal{\J}_{22}^{\G}} \not\subseteq \overline{\mathcal{N}_{27}^{\#}}  $  and therefore $\J_{22}$ is rigid.
To show that the other three algebras are rigid it is sufficient to prove that there 
are no deformations between them and that 
none of these algebras is dominated by the family $\J_{27}^{\varepsilon,\phi}$. 
 
Using the dimension argument for the second power of an algebra we have $\J_{21}\nrightarrow\epsilon_{1}$,
$\J_{22}\nrightarrow\epsilon_{1}$ and $\J_{21}\nrightarrow\J_{40}$, while comparing the 
dimension of the annihilator we obtain $\J_{40}\nrightarrow\epsilon_{1}$. Thus we have shown that
$\epsilon_{1}$ is a rigid algebra in $\JorN_{5}$ too. The associative algebra 
$\epsilon_{1}$ can not dominate non-associative algebras, thus 
$\epsilon_{1}\nrightarrow\J_{21}$ and $\epsilon_{1}\nrightarrow\J_{40}$. Also, due
to the dimension argument for the automorphism group we get 
$\J_{22}\nrightarrow\J_{40}$ and comparing the dimension of the associative center we have  $\J_{22}\nrightarrow\J_{21}$ and $\J_{40}\nrightarrow\J_{21}$. On the other
hand $\dim\J_{21}^{\G}=22$ while $\dim\overline{\mathcal{N}_{27}^{\#}}=21$
leading to $\J_{21}^{\G}\nsubseteq\overline{\mathcal{N}_{27}^{\#}}$ and therefore
$\J_{21}$ is a rigid algebra. Finally suppose that  $\J_{40}$ belongs to the 
component $\overline{\mathcal{N}_{27}^{\#}}$. Note 
that the dimensions of $\J_{40}^{\G}$ and  $\mathcal{N}_{27}^{\#}$ are both equal to $21$,
therefore 
$$\overline{\J_{40}^{\G}}=\overline{\mathcal{N}_{27}^{\#}}.$$
Which is impossible since $\dim\Ann(\J_{40})=2$, while 
$\dim\Ann(\J_{27}^{\varepsilon,\phi})=1$ for $\varepsilon,\phi \in\ka$ with 
$\varepsilon\phi\neq1$. Consequently $\J_{40}$ is a rigid algebra.
\end{proof}

Finally, note that the change of basis given by identity matrix $I_5$ multiplied by $t$
guarantees $\J\to\epsilon_{25}$ for any $\J\in\JorN_{5}$. Consequently, $\JorN_{5}$ is a connected affine
variety. Therefore we have proved the following theorem.
\begin{thm}
The algebraic variety $\JorN_{5}$ of five-dimensional nilpotent Jordan
algebras is a connected affine variety with an infinite
number of orbits under the action of the group $\G$. Furthermore,
it has four irreducible components given by Zariski closures of the
orbits of the rigid algebras $\epsilon_{1}$, $\J_{21}$,
$\J_{22}$ and $\J_{40}$ and one more component that is the Zariski
closure of the infinite union $\mathcal{N}_{27}^{\#}$.
\end{thm}
The diagram of $\JorN_{5}$ was represented in Figure \ref{fig:orbitas}.

\begin{figure}  \centering \includegraphics[trim={1.2cm 0cm 0cm 2.7cm}, 
clip=true, angle=90,width=0.60\textwidth]{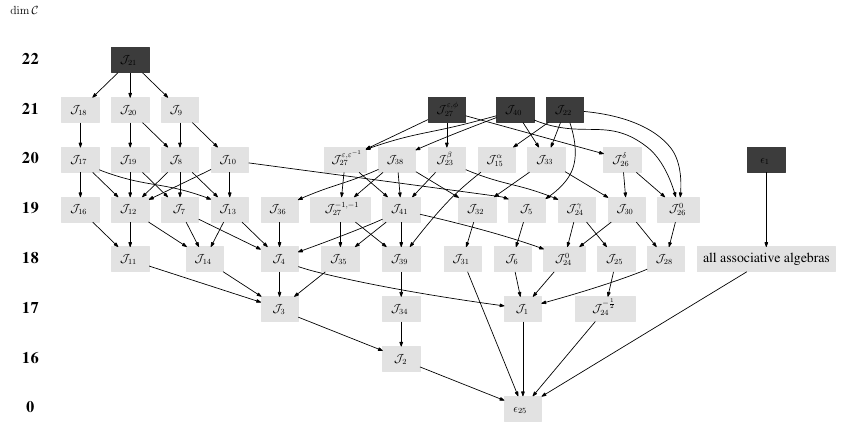}\caption{ The irreducible 
components of $\JorN_5$\label{fig:orbitas} } \end{figure}

\newpage

\bibliographystyle{elsarticle-num.bst}
\bibliography{library}

\end{document}